\documentclass[12pt,reqno]{amsart}

\usepackage{amssymb}

\usepackage[all]{xy}

\usepackage{hyperref}
\hypersetup{colorlinks,
linkcolor=black,
filecolor=black,
urlcolor=black,
citecolor=black}

\numberwithin{equation}{section}

\newtheorem{theorem}{Theorem}[section]
\newtheorem{proposition}[theorem]{Proposition}
\newtheorem{lemma}[theorem]{Lemma}

\theoremstyle{definition}
\newtheorem{definition}[theorem]{Definition}
\newtheorem{remark}[theorem]{Remark}

\begin{document}

\baselineskip=15pt

\title[Symplectic and orthogonal parabolic bundles]{Pullback and direct image of parabolic Higgs bundles and 
parabolic connections with symplectic and orthogonal structures}

\author[D. Alfaya]{David Alfaya}

\address{Department of Applied Mathematics and Institute for Research in Technology, ICAI School of Engineering,
Comillas Pontifical University, C/Alberto Aguilera 25, 28015 Madrid, Spain}

\email{dalfaya@comillas.edu}

\author[I. Biswas]{Indranil Biswas}

\address{Department of Mathematics, Shiv Nadar University, NH91, Tehsil Dadri,
Greater Noida, Uttar Pradesh 201314, India}

\email{indranil.biswas@snu.edu.in, indranil29@gmail.com}

\author[F.-X. Machu]{Francois-Xavier Machu}

\address{ESIEA, 74 bis Av. Maurice Thorez, 94200 Ivry-sur-Seine, France}

\email{fx.machu@gmail.com}



\date{}

\begin{abstract}
Given a symplectic (respectively, orthogonal) parabolic vector bundle over a compact Riemann surface, we 
prove that its pullback and direct image through a map between compact Riemann surfaces inherit a natural symplectic 
(respectively, orthogonal) structure. If the parabolic bundle is endowed with a parabolic Higgs field or a 
parabolic connection which are compatible with the symplectic (respectively, orthogonal) structure, then 
its pullback and direct image are also compatible with the resulting symplectic (respectively, orthogonal) 
structure. We also show that these constructions are preserved through the Nonabelian Hodge Correspondence.
\end{abstract}

\maketitle

\section{Introduction}

Let $X$ be a smooth complex projective curve with a subset of marked distinct $n$ points $D\,=\,\{x_1,\,\cdots,\,x_n\}
\, \subset\, X$. A parabolic vector bundle on $(X,\,D)$ is a holomorphic vector bundle $E$
on $X$ endowed with a weighted filtration of the fiber $E_{x_i}$ over each parabolic point $x_i\,\in \,D$,
which is called a parabolic structure. This notion was first introduced by C. S. Seshadri \cite{Se}
in the context of desingularization of moduli spaces of vector bundles. Parabolic bundles play
a key role in the nonabelian Hodge correspondence for noncompact curves \cite{Si}. The notion of parabolic
structure was generalized to $G$-bundles in \cite{BBN}, where $G$ is a reductive
affine algebraic group. In this work we will focus on the case where $G$ is either
the symplectic group or the orthogonal group. In this case, given a parabolic line bundle $L_*$, an $L_*$-valued
parabolic symplectic (respectively, orthogonal) bundle is a parabolic vector bundle $E_*$ endowed with a bilinear map
$$\phi\, : \, E_* \otimes E_* \,\longrightarrow\, L_*$$
which is antisymmetric (respectively, symmetric) satisfying the condition that the map of parabolic bundles
$E_* \,\longrightarrow \, E_*^* \otimes L_*$ induced by adjunction is an isomorphism (see Section \ref{se2}
for details).

In this paper we study the stability of parabolic symplectic and parabolic orthogonal bundles, as well as their 
pullbacks and direct images through nonconstant maps of Riemann surfaces. In \cite{AB} it was shown that the 
pullback and direct image of a parabolic vector bundle inherit a natural parabolic structure.
It was also shown there that such a 
structure is compatible with the pullbacks and direct images of logarithmic Higgs fields, logarithmic 
connections and the Nonabelian Hodge Correspondence for noncompact curves. In this work, we will extend those 
results to parabolic symplectic and parabolic orthogonal structures and prove that the pullback and direct image 
of parabolic vector bundles, parabolic Higgs bundles and parabolic connections with a symplectic
(respectively, orthogonal) structure also inherit natural symplectic (respectively, orthogonal)
structures. We prove that taking pullbacks and direct images of these objects preserve semistability and
polystability and that these constructions are preserved by the Nonabelian Hodge Correspondence for
noncompact curves.

The paper is structured as follows. Parabolic orthogonal bundles and parabolic symplectic bundles are reviewed 
in Section \ref{se2}, and some properties on the semistability and polystability of parabolic bundles with 
symplectic and orthogonal structures are deduced. Parabolic Higgs bundles and parabolic connections compatible 
with symplectic or orthogonal structures and their stability are studied in Section \ref{se3}. The pullback of 
parabolic bundles, parabolic Higgs bundles and parabolic connections endowed with a symplectic or orthogonal 
structure is described in Section \ref{section:pullback}, and the direct images of such structures are described 
through Section \ref{section:pushforward}. Finally, in Section \ref{section:NAHC}, it is proven that the 
noncompact Nonabelian Hodge Correspondence preserves all the described constructions.

\subsection{A brief description of the strategy}

Take an $n$--pointed Riemann surface $(X,\, D)$ and a nonconstant holomorphic map $f\,:\, Y\, \longrightarrow\,
X$ from a compact Riemann surface $Y$. As mentioned before, for any parabolic vector bundle $E_*$ on $(X,\, D)$,
we have a pulled back parabolic vector bundle $f^*E_*$ on $Y$ with parabolic structure on the reduced divisor
$f^{-1}(D)_{\rm red}$. If $D$ is disjoint from the divisor on $X$ over which $f$ is branched, then
the underlying vector bundle for $f^*E_*$ is simply $f^*E$, where $E$ is the vector bundle on $X$ underlying
$E_*$. This pullback operation is compatible with respect to the operations of direct sum and tensor product
of parabolic vector bundles. Also, $f^*E^*_*$ coincides with the parabolic dual of $f^*E_*$. Furthermore,
a homomorphism of parabolic bundles $h\,:\, E_*\, \longrightarrow\, V_*$ pulls back to a homomorphism of
parabolic bundles $f^* h\,:\, f^*E_*\, \longrightarrow\, f^*V_*$. Using these properties of parabolic
pullback, a symplectic (respectively, orthogonal) structure $\phi$ on $E_*$ pulls back to a symplectic
(respectively, orthogonal) structure $f^*\phi$ on $f^* E_*$. Furthermore, a symplectic (respectively,
orthogonal) parabolic Higgs field on $(E_*,\, \phi)$ pulls back to a symplectic (respectively,
orthogonal) parabolic Higgs field on $(f^*E_*,\, f^*\phi)$. The same holds for parabolic connections.

The picture is a little more complicated for direct images. Let $\Phi\, :\, X\, \longrightarrow\, Z$ be a 
nonconstant holomorphic map, where $Z$ is a compact Riemann surface. Given a parabolic vector bundle $E_*$ 
on $(X,\, D)$ we have a parabolic vector bundle $\Phi_*E_*$ on $Z$ with parabolic structure on the union of 
$f(D)$ and the divisor on $Z$ over which $\Phi$ is branched. The vector bundle underlying $\Phi_*E_*$ is 
simply $\Phi_*E$, where $E$ as before is the vector bundle on $X$ underlying $E_*$. As before, we have 
$\Phi_*(E_*\oplus V_*) \,=\, (\Phi_*E_*)\oplus (\Phi_*V_*)$ and $\Phi_* E^*_*\, =\, (\Phi_* E_*)^*$, but 
$\Phi_*(E_*\otimes V_*)$ is not the same as $(\Phi_*E_*)\otimes (\Phi_*V_*)$ (they have different ranks). 
But there is a natural parabolic homomorphism from $(\Phi_*E_*)\otimes (\Phi_*V_*)$ to $\Phi_*(E_*\otimes 
V_*)$. This enables us to construct parabolic symplectic (respectively, orthogonal) structure on 
$\Phi_*E_*$ given a parabolic symplectic (respectively, orthogonal) structure on $E_*$.

\section{Parabolic orthogonal and symplectic bundles}\label{se2}

Let $X$ be a compact connected Riemann surface. The holomorphic cotangent bundle of $X$ will be denoted by $K_X$.
Let
$$
D\,\,:=\,\, \{x_1,\, \cdots,\, x_\ell\}\,\, \subset\,\, X
$$
be a finite subset. The reduced effective divisor $\sum_{i=1}^\ell x_i$ on $X$ will also be denoted by
$D$. For a holomorphic vector bundle $V$ on $X$, the vector bundles $V\otimes {\mathcal O}_X(D)$ and
$V\otimes {\mathcal O}_X(-D)$ on $X$ will be denoted by $V(D)$ and $V(-D)$ respectively.

Take a holomorphic vector bundle $E$ on $X$. A \textit{quasi-parabolic structure} on $E$ is a
strictly decreasing filtration of subspaces
\begin{equation}\label{e1}
E_{x_i}\,=\, E^1_i\, \supsetneq\, E^2_i \,\supsetneq\, \cdots\, \supsetneq\,
E^{n_i}_i \, \supsetneq\, E^{n_i+1}_i \,=\, 0
\end{equation}
for every $1\, \leq\, i\, \leq\, \ell$; here $E_{x_i}$ denotes the fiber
of $E$ over the point $x_i\,\in\, D$. A \textit{parabolic structure} on $E$ is a
quasi-parabolic structure as above together with $\ell$ increasing sequences of real numbers
\begin{equation}\label{e2}
0\, \leq\, \alpha_{i,1}\, <\, \alpha_{i,2}\, <\,
\cdots\, < \, \alpha_{i,n_i}\, < 1\, , \ \ 1\, \leq\, i\, \leq\, \ell \, ;
\end{equation}
the number $\alpha_{i,j}$ in \eqref{e2} is called the parabolic weight of the subspace $E^j_i$ occurring in
the quasi-parabolic filtration in \eqref{e1}. For any $1\,\leq\, j\, \leq\, n_i$, the multiplicity of a parabolic
weight $\alpha_{i,j}$ at $x_i$ is defined to be the dimension of the complex vector space $E^j_i/E^{j+1}_i$.
A \textit{parabolic vector bundle} is a holomorphic vector bundle equipped with a parabolic structure.

The \textit{parabolic degree} of a parabolic vector bundle $$E_*\,:=\, \left(E,\,
\{\{E^j_i\}_{j=1}^{n_i}\}_{i=1}^\ell, \,\{\{\alpha_{i,j}\}_{j=1}^{n_i}\}_{i=1}^\ell\right)$$
is defined to be
$$
\text{par-deg}(E_*)\,=\, \text{degree}(E)+\sum_{i=1}^\ell \sum_{j=1}^{n_i}
\alpha_{i,j}\cdot\dim \left(E^j_i/E^{j+1}_i\right)
$$
\cite[p.~214, Definition~1.11]{MS}, \cite[p.~78]{MY}. The real number
$$
\text{par-}\mu(E_*)\,\, :=\, \frac{\text{par-deg}(E_*)}{\text{rank}(E_*)}
$$
is called the parabolic slope of $E_*$.

See \cite{Yo}, \cite{Bi1} for the operations of parabolic direct sum, parabolic tensor product and parabolic dual.
The holomorphic sections of a parabolic vector bundle $E_*$ are, by definition, the holomorphic sections of
the underlying vector bundle $E$. For parabolic vector bundles $E_*$ and $E'_*$, the global homomorphisms
$E_*\, \longrightarrow\, E'_*$ are identified with the holomorphic sections of the parabolic vector bundle $E'_*\otimes E^*_*$.
For parabolic vector bundles $E_*$ and $E'_*$, the restrictions $(E_*\oplus E'_*)\big\vert_{X\setminus D}$
and $(E_*\otimes E'_*)\big\vert_{X\setminus D}$ are identified with $(E\oplus E')\big\vert_{X\setminus D}$
and $(E\otimes E')\big\vert_{X\setminus D}$ respectively. The vector bundle underlying the parabolic
vector bundle $E_*\otimes E^*_*\,=\, \text{End}(E_*)_*$ coincides with the subsheaf
\begin{equation}\label{e2b}
\text{End}_P(E)\,\, \subset\,\,\text{End}(E)
\end{equation}
defined by all $s\,\in\, \Gamma(U,\, \text{End}(E))$, where $U\, \subset\, X$ is any open
subset, such that $s(E^j_i)\, \subset\, E^j_i$ for all $x_i\, \in\, U$ and $1\, \leq\, j\, \leq\, n_i$.

Consider the subbundle ${\rm ad}(E)\, \subset\, \text{End}(E)$ of co-rank one given by the sheaf of endomorphisms
of $E$ of trace zero. Let
$$
\text{ad}_P(E)\ :=\ \text{End}_P(E)\cap \text{ad}(E) \ \subset\ \text{End}(E)
$$
be the intersection. We note that $\text{End}_P(E)$ decomposes as
\begin{equation}\label{e3}
\text{End}_P(E)\,\, = \,\, {\mathcal O}_X\cdot {\rm Id}_E \oplus (\text{End}_P(E)\cap \text{ad}(E))
\,\,=\,\, {\mathcal O}_X\cdot {\rm Id}_E \oplus \text{ad}_P(E).
\end{equation}

Assume that $E_*$ has a nontrivial parabolic structure at every $x_i$, $1\, \leq\, i\, \leq\, \ell$; this means that
$\alpha_{i,n_i} \, >\, 0$ for all $1\, \leq\, i\, \leq\, \ell$. Let
\begin{equation}\label{e2c}
\text{End}^0_P(E)\,\, \subset\,\,\text{End}_P(E)
\end{equation}
be the subsheaf of $\text{End}_P(E)$ (see \eqref{e2b}) defined by all $s\,\in\, \Gamma(U,\, \text{End}_P(E))$,
where $U\, \subset\, X$ is any open subset, such that $s(E^j_i)\, \subset\, E^{j+1}_i$ for all $x_i\, \in\, U$
and $1\, \leq\, j\, \leq\, n_i$. Denoting $\text{End}^0_P(E)\cap \text{ad}(E)$ by $\text{ad}^0_P(E)$, we have
\begin{equation}\label{e3b}
\text{End}^0_P(E)\,\, = \,\, {\mathcal O}_X(-D)\cdot {\rm Id}_E \oplus (\text{End}^0_P(E)\cap \text{ad}(E))
\,\,=\,\, {\mathcal O}_X(-D)\cdot {\rm Id}_E \oplus \text{ad}^0_P(E).
\end{equation}

Let $L_*$ be a parabolic line bundle on $X$ of parabolic degree $0$, and let
\begin{equation}\label{e4}
\phi\,\,:\,\, E_*\otimes E_*\,\, \longrightarrow\,\, L_*
\end{equation}
be a homomorphism of parabolic vector bundles. Tensoring \eqref{e4} by the parabolic dual $E^*_*$, we have
$$
\phi\otimes {\rm Id}_{E^*_*} \ :\ (E_*\otimes E_*)\otimes E^*_* \ =\ E_*\otimes (E_*\otimes E^*_*)
$$
$$
=\ E_*\otimes\text{End}(E_*)_* \ \longrightarrow\ L_*\otimes E^*_*\,\cong\, E^*_*\otimes L_*.
$$
Restricting this homomorphism to $E_*\,=\, E_*\otimes {\mathcal O}_X\, \subset\, E_*\otimes\text{End}(E_*)_*$,
(the parabolic structure on ${\mathcal O}_X$ is the trivial one; see the decomposition in \eqref{e3}) we
get a homomorphism
\begin{equation}\label{e5}
\widehat{\phi}\,\,:\,\, E_*\,\, \longrightarrow\,\, E^*_* \otimes L_*.
\end{equation}

\begin{definition}\label{def1}
The pairing $\phi$ in \eqref{e4} defines an {\it $L_*$-valued orthogonal structure} on $E_*$ if the pairing
$$
\phi\big\vert_{X\setminus D}\,\,:\,\, (E\big\vert_{X\setminus D})\otimes (E\big\vert_{X\setminus D})
\,\, \longrightarrow\,\, L\big\vert_{X\setminus D}
$$
is symmetric and $\widehat{\phi}$ in \eqref{e5} is an isomorphism of parabolic vector bundles.
For notational simplicity, when $L_*$ is clear form the context, or when $L_*$ is the trivial line
bundle $\mathcal{O}_X$ with the trivial parabolic structure, we will omit the mention to $L_*$ and call
$\phi$ simply an {\it orthogonal structure} on $E_*$.

The pairing $\phi$ defines an {\it $L_*$-valued symplectic structure} on $E_*$ if the above restriction 
$\phi\big\vert_{X \setminus D}$ is anti-symmetric and $\widehat{\phi}$ in \eqref{e5} is an isomorphism of 
parabolic vector bundles. Like as before, when $L_*$ is clear from the context or it is trivial, we call
$\phi$ a {\it symplectic structure} on $E_*$.

If $\phi$ is an $L_*$-valued parabolic symplectic (respectively, parabolic orthogonal) structure on $E_*$, then
the pair $(E_*,\,\phi)$ is called a {\it parabolic symplectic} (respectively, {\it parabolic orthogonal}) bundle.
\end{definition}

\begin{lemma}
\label{lemma:transposition}
Let $\tau\,:\, E_* \otimes E_* \,\longrightarrow\, E_*\otimes E_*$ be defined by $u\otimes v\, \longmapsto\,
v\otimes u$. Then a map $\phi \,:\, E_*\otimes E_* \,\longrightarrow\, L*$ such that
$\widehat\phi \,:\, E_* \,\longrightarrow\, E_*^*\otimes L_*$ is an isomorphism is an orthogonal structure if
and only if $\phi \circ \tau \,=\, \phi$, and it is a symplectic structure if and only if $\phi \circ \tau
\,=\, -\phi$.
\end{lemma}

\begin{proof}
By Definition \ref{def1}, $\phi$ is an orthogonal structure if and only $(\phi\circ \tau-
\phi)\big\vert_{X\setminus D}\,=\,0$. The section $\phi \circ \tau-\phi \,\in\,
H^0(\text{Hom}(E_*\otimes E_*,\, L_*))$ vanishes over $X\setminus D$ if and only if it vanishes
on $X$, so $\phi$ is an orthogonal structure if and only if $\phi\circ \tau-\phi\,=\,0$. The proof for
the symplectic case is analogous.
\end{proof}

\begin{lemma}\label{lem1}
For any parabolic symplectic (respectively, parabolic orthogonal) vector bundle $(E_*,\,\phi)$,
$${\rm par}\text{-}{\rm deg}(E^*_*)\ =\ 0.$$
\end{lemma}

\begin{proof}
Since $\text{par-deg}(E^*_*)\,=\, - \text{par-deg}(E_*)$, if $\widehat\phi$ in \eqref{e5} is an
isomorphism and also $\text{par-deg}(L_*)\, =\, 0$, then we have
$$\text{par-deg}(E_*)\,=-\text{par-deg}(E_*)+\text{rank}(E) \text{par-deg}(L_*) \, = \, -\text{par-deg}(E_*),$$
so $\text{par-deg}(E_*)\,=\, 0$. As $\widehat\phi$ is an isomorphism for any parabolic
symplectic or orthogonal vector bundle $(E_*,\, \phi)$, the lemma follows.
\end{proof}

Take a parabolic symplectic or parabolic orthogonal vector bundle $(E_*,\, \phi)$. It is called
\textit{stable} (respectively, \textit{semistable}) if
$$
\text{par-deg}(F_*)\, \, <\,\, 0\ \ \, \text{(respectively, }\ \text{par-deg}(F_*)\, \, \leq\,\, 0\text{)}
$$
for every subbundle $0\, \not=\, F\, \subsetneq\, E$ such that $\phi\big\vert_{X\setminus D}
((F\big\vert_{X\setminus D})\otimes (F\big\vert_{X\setminus D}))\,=\, 0$, where $E$ is the holomorphic
vector bundle on $X$ underlying $E_*$, and $F_*$ is the parabolic vector bundle defined by the
subbundle $F$ equipped with the parabolic structure on it induced by $E_*$.

The following observation is needed in the definitions of polystable parabolic orthogonal and polystable parabolic
symplectic vector bundles.

Let $E_*$ be a parabolic vector bundle on $X$. Then the parabolic vector bundle $E_*\oplus (E^*_*\otimes L)$
has a natural
orthogonal structure as well as a natural symplectic structure. To see this, first note that
$$(E_*\oplus (E^*_*\otimes L_*))^*\otimes L_* \,=\, (E^*_* \oplus (E_*\otimes L_*^*))\otimes L_*
\,=\, (E^*_*\otimes L_*) \oplus E_*$$
Consider the natural isomorphisms
\begin{gather}\begin{aligned}\label{e6}
\widehat{\phi_{\text{orth}}}\ = \ \begin{pmatrix}
0 & \text{Id}_{E_*^*\otimes L_*}\\
\text{Id}_{E_*} & 0
\end{pmatrix} \ : & \\
E_*\oplus (E_*^*\otimes L_*)\ \longrightarrow\
(E_*^*\otimes L_*) \oplus E_*\ & =\ (E_*\oplus (E_*^*\otimes L_*))^*\otimes L_*\\
\widehat{\phi_{\text{symp}}}\ =\ \begin{pmatrix}
0 & \text{Id}_{E_*^*\otimes L_*}\\
-\text{Id}_{E_*} & 0
\end{pmatrix} \ : & \\
\, E_*\oplus (E_*^*\otimes L_*)\ \longrightarrow\ (E_*^*\otimes L_*) \oplus E_* 
\ & =\ (E_*\oplus (E_*^*\otimes L_*))^*\otimes L_*
\end{aligned}\end{gather}
Then, $\widehat{\phi_{\text{orth}}}$ and $\widehat{\phi_{\text{symp}}}$ induce by adjunction maps
$$\phi_{\text{orth}},\, \, \phi_{\text{symp}} \, :\,
(E_*\oplus (E_*^*\otimes L_*)) \otimes (E_*\oplus (E_*^*\otimes L_*)) \longrightarrow L_*.$$
By construction, $\phi_{\text{orth}}$ is symmetric and $\phi_{\text{symp}}$ is anti-symmetric, thus yielding an orthogonal and a symplectic structure respectively.

A parabolic symplectic (respectively, parabolic orthogonal) vector bundle $(E_*,\, \phi)$ is called \textit{polystable}
if the following two conditions hold:
\begin{enumerate}
\item $(E_*,\, \phi)$ is semistable, and

\item $(E_*,\, \phi)\,=\, \bigoplus_{i=1}^n (V_{i,*},\, \phi_i)$, where each $(V_{i,*},\, \phi_i)$ is either a
stable parabolic symplectic (respectively, parabolic orthogonal) vector bundle or there is a stable
parabolic vector bundle $W_{i,*}$ of parabolic degree zero such that $(V_{i,*},\, \phi_i)$ is isomorphic to
$W_{i,*}\oplus (W_{i,*}^*\otimes L_*)$ equipped with the above natural parabolic symplectic (respectively,
parabolic orthogonal) pairing.
\end{enumerate}
The following result allows us to relate the semistability and polystability of a parabolic symplectic or orthogonal bundle with the semistability and polystability of its underlying parabolic vector bundle.

\begin{proposition}\label{prop0}
Let $(E_*,\, \phi)$ be a parabolic symplectic (respectively, parabolic orthogonal) vector bundle on $X$.
\begin{enumerate}
\item The parabolic symplectic (respectively, parabolic orthogonal) vector bundle $(E_*,\, \phi)$ is semistable
if and only if the parabolic vector bundle $E_*$ is semistable.

\item The parabolic symplectic (respectively, parabolic orthogonal) vector bundle $(E_*,\, \phi)$ is polystable
if and only if the parabolic vector bundle $E_*$ is polystable.
\end{enumerate}
\end{proposition}

\begin{proof}
If the parabolic vector bundle $E_*$ is semistable, then it is evident that the parabolic symplectic (respectively,
parabolic orthogonal) vector bundle $(E_*,\, \phi)$ is semistable.

To prove the converse, assume that
the parabolic symplectic (respectively, parabolic orthogonal) vector bundle $(E_*,\, \phi)$ is semistable. To prove
that the parabolic vector bundle $E_*$ is semistable using contradiction, assume that $E_*$ is not semistable. Let
\begin{equation}\label{h1}
0\,=\, E^0_*\, \subset\, E^1_* \, \subset\, E^2_*\, \subset\, \cdots\, \subset\, E^{\ell-1}_*\, \subset\,
E^{\ell}_* \, =\, E_*
\end{equation}
be the Harder--Narasimhan filtration of $E_*$ (see \cite{HL} for Harder--Narasimhan filtration). Note that $\ell \,
\geq\, 2$ because it is assumed that $E_*$ is not semistable. For a parabolic subbundle $\iota \, :\, W_*\, \hookrightarrow
\, E_*$, the kernel of the parabolic dual homomorphism $\iota^*\, :\, E^*_* \, \longrightarrow\, W^*_*$ will be denoted by
$(W_*)^\perp$. From the general properties of the Harder--Narasimhan filtration it follows immediately that
$$
0\,=\, (E^{\ell}_*)^\perp \, \subset\, (E^{\ell -1}_*)^\perp \, \subset\, (E^{\ell- 2}_*)^\perp \, \subset\, \cdots\, \subset\,
(E^2_*)^\perp \, \subset\, (E^1_*)^\perp \, \subset\, (E^0_*)^\perp \, =\, E^*_*
$$
is the Harder--Narasimhan filtration of the parabolic dual $E^*_*$ (see \eqref{h1}). Since tensoring
with a parabolic line bundle preserves semistability,
\begin{equation}\label{h2b}
0\,=\, (E^{\ell}_*)^\perp\otimes L_* \, \subset\, (E^{\ell -1}_*)^\perp\otimes L_* \, \subset\, \cdots\, \subset\,
 (E^1_*)^\perp\otimes L_* \, \subset\, (E^0_*)^\perp\otimes L_* \, =\, E^*_*\otimes L_*
\end{equation}
is the Harder--Narasimhan filtration of $E_*^*\otimes L_*$. On the other hand, $\widehat{\phi}$
in \eqref{e5} is an isomorphism $E_*\, \stackrel{\sim}{\longrightarrow}\, E^*_*\otimes L_*$. Thus, from
the uniqueness of the Harder--Narasimhan filtration we have
$$
\widehat{\phi}(E^1_*)\ =\ (E^{\ell -1}_*)^\perp\otimes L_* \ \subseteq\ (E^1_*)^\perp \otimes L^*,
$$
but this implies that $\phi\big\vert_{X\setminus D} ((E^1_*\big\vert_{X\setminus D})\otimes 
(E^1_*\big\vert_{X\setminus D}))\,=\, 0$. Consequently, the parabolic subbundle $E^1_*\, \subset\, E_*$ 
contradicts the given condition that the parabolic symplectic (respectively, parabolic orthogonal) vector 
bundle $(E_*,\, \phi)$ is semistable.

In view of the above contradiction we conclude that the parabolic vector
bundle $E_*$ is semistable. This proves statement (1) of the proposition.

To prove (2), first assume that the parabolic symplectic (respectively, parabolic orthogonal) vector bundle $(E_*,\, \phi)$
is polystable. In particular, $(E_*,\, \phi)$ is semistable. From (1) we know that the parabolic vector bundle $E_*$ is
semistable. To prove that $E_*$ is polystable using contradiction, assume that $E_*$ is not polystable. Let
\begin{equation}\label{h3}
0\,=\, E^0_*\, \subset\, E^1_* \, \subset\, E^2_*\, \subset\, \cdots\, \subset\, E^{\ell-1}_*\, \subset\,
E^{\ell}_* \, =\, E_*
\end{equation}
be the socle filtration of $E_*$ (see \cite[p.~23, Lemma 1.5.5]{HL} for the socle filtration); so for
any $1\, \leq\, i\,\leq\, \ell-1$, the parabolic quotient $E^i_*/E^{i-1}_*$ is the unique maximal polystable
subbundle of $E_*/E^{i-1}_*$ with $\text{par-}\mu(E^i_*/E^{i-1}_*)\,=\, \text{par-}\mu(E_*)\,=\, 0$. We have
$\ell \,\geq\, 2$ because of the assumption that the parabolic semistable vector bundle $E_*$ is not
polystable. The
parabolic dual $E^*_*$ is semistable because $E_*$ is so. Hence $E^*_*$ has a socle filtration. Also, the dual
of a parabolic polystable vector bundle is again parabolic polystable. Hence from
the general properties of the socle filtration it follows that
$$
0\,=\, (E^{\ell}_*)^\perp \, \subset\, (E^{\ell -1}_*)^\perp \, \subset\, (E^{\ell- 2}_*)^\perp \, \subset\, \cdots\, \subset\,
(E^2_*)^\perp \, \subset\, (E^1_*)^\perp \, \subset\, (E^0_*)^\perp \, =\, E^*_*
$$
is the socle filtration of the parabolic dual $E^*_*$ (see \eqref{h3}). Since tensoring with a parabolic
line bundle preserves both semistability and polystability,
\begin{equation}\label{h4b}
0\,=\, (E^{\ell}_*)^\perp\otimes L_* \, \subset\, (E^{\ell -1}_*)^\perp\otimes L_* \, \subset\, \cdots\, \subset\,
(E^1_*)^\perp\otimes L_* \, \subset\, (E^0_*)^\perp\otimes L_* \, =\, E^*_*\otimes L_*
\end{equation}
is the socle filtration of the parabolic tensor product $E_*^*\otimes L_*$. Recall that $\widehat{\phi}\, :\,
E_*\, \longrightarrow\, E^*_* \otimes L_*$ in \eqref{e5} is an isomorphism. So from the uniqueness of the socle
filtration it follows immediately that
\begin{equation}\label{h5}
\widehat{\phi}(E^1_*)\ =\ (E^{\ell -1}_*)^\perp\otimes L_* \ \subseteq\ (E^1_*)^\perp\otimes L_*.
\end{equation}
so $\phi(E^1_*\otimes E^1_*)\,=\,0$. Since $(E_*,\, \phi)$ is polystable, and $\text{par-}\mu(E^1_*)\,=\, 0$, from \eqref{h5} it follows that $E^1_*\, \subset\, E_*$
has a direct summand. Since $E_*$ is parabolic semistable of parabolic degree zero, the a direct
summand is again parabolic semistable of parabolic degree zero. 
Adding to $E^1_*$ the socle of a direct summand of $E^1_*$ we get a polystable subbundle of $E_*$ which
strictly contains $E^1_*$. But this contradicts the fact that $E^1_*$ is the unique maximal polystable parabolic subbundle of
$E_*$ of parabolic degree zero. In view of this contradiction we conclude that the
parabolic vector bundle $E_*$ is parabolic polystable.

To prove the converse, assume that the parabolic vector bundle $E_*$ is polystable. Since 
$\text{par-deg}(E_*)\,=\, 0$ (see Lemma \ref{lem1}), we know that $E_*$ given by a homomorphism
$$
\rho\ :\ \pi_1(X\setminus D,\, x_0)\ \longrightarrow\ \text{U}(r),
$$
where $r\,=\, {\rm rank}(E_*)$ and $x_0\, \in\, X\setminus S$ is a base point (see \cite{MS}). So
$E_*\otimes E_*$ is given by the representation
$$
\rho\otimes\rho \ :\ \pi_1(X\setminus D,\, x_0)\ \longrightarrow\ \text{U}({\mathbb C}^r\otimes{\mathbb C}^r)
\ =\ \text{U}(r^2).
$$
Hence the holomorphic vector bundle $(E_*\otimes E_*)_0$ underlying the parabolic tensor product $E_*
\otimes E_*$ is equipped with a logarithmic connection singular over $D$ \cite{MS}; the definition of a
logarithmic connection is recalled in Section \ref{se3.1}. This logarithmic
connection on $(E_*\otimes E_*)_0$ given by $\rho\otimes\rho$ will be denoted by $\mathbf\nabla$.

There is a unique unitary logarithmic connection $D_{L}\,:\, L
\,\longrightarrow\, L\otimes K_X(D)$ such that ${\rm Res}(D_L,\,x_i)\,=\,\beta_i$
for all $1\, \leq\, i\, \leq\, \ell$ (see \cite{De}, \cite{Ka}).
The homomorphism $\phi\,:\, E_*\otimes E_*\, \longrightarrow\, L_*$ (see \eqref{e4}) is flat (same as integrable)
with respect to the connection $D_L$ on $L_*$ and the logarithmic connection $\mathbf\nabla$
on $E_*\otimes E_*$; this is because the parabolic slopes of $E_*\otimes E_*$ and $L_*$ coincide. This implies that
the logarithmic connection on $E_*$ preserves the bilinear form $\phi$. Consequently, $(E_*,\, \phi)$ is polystable.
\end{proof}

\begin{remark}\label{rem1}
Proposition \ref{prop0} does not extend to stable parabolic bundles. More precisely, if the parabolic 
symplectic (respectively, parabolic orthogonal) vector bundle $(E_*,\, \phi)$ is stable, then the parabolic 
vector bundle $E_*$ need not be stable. To give an example, let $(E^1_*,\, \phi_1)$ and $(E^2_*,\, \phi_2)$ 
be parabolic symplectic (respectively, parabolic orthogonal) vector bundles such that $E_1$ and and $E_2$ 
are both stable with $E^1_*\, \not=\, E^2_*$. Then $(E^1_*\oplus E^2_*,\, \phi_1\oplus\phi_2)$ is a stable 
parabolic symplectic (respectively, parabolic orthogonal) vector bundle, but $E^1_*\oplus E^2_*$ is not 
stable.

To give the simplest example of the above type, let $L$ be a nontrivial holomorphic line bundle on $X$ of order two.
Note that both ${\mathcal O}_X$ and $L$ have a natural orthogonal structure. Consider the orthogonal structure on
$L\oplus {\mathcal O}_X$ obtained by taking the direct sum of the orthogonal structures on $L$ and ${\mathcal O}_X$.
Then the resulting rank two orthogonal bundle is stable, but the underlying vector bundle $L\oplus {\mathcal O}_X$
is not stable.

A similar result to Proposition \ref{prop0} can be obtained from \cite[Proposition 5.6, Proposition 5.7 and 
Corollary 6.2]{BMW}, using different techniques to obtain the equivalence between the polystability of 
parabolic orthogonal/symplectic bundles and the polystability of their underlying parabolic vector bundles.
\end{remark}

\section{Higgs bundles and connections}\label{se3}

\subsection{Parabolic connections}\label{se3.1}

Let $V$ be a holomorphic vector bundle on $X$. A \textit{logarithmic connection} on
$V$ singular over $D$ is a holomorphic differential operator
$$
{\nabla}\, :\, V\, \longrightarrow\, V\otimes K_X(D)
$$
satisfying the Leibniz identity which states that
\begin{equation}\label{f4}
{\nabla}(fs)\,=\, f{\nabla}(s)+ s\otimes df
\end{equation}
for any locally
defined holomorphic function $f$ on $X$ and any locally defined holomorphic section $s$ of $V$ (see
\cite{De}, \cite{At}). So a logarithmic connection on $V$ produces a holomorphic connection
$V\big\vert_{X\setminus D}$.

For any $y\, \in\, D$, the Poincar\'e adjunction formula gives an isomorphism
\begin{equation}\label{pa}
K_X(D)_y \, \stackrel{\sim}{\longrightarrow}\, {\mathbb C}
\end{equation}
(see \cite[p.~146]{GH} for the Poincar\'e adjunction formula).
Let ${\nabla}^V\, :\, V\, \longrightarrow\, V\otimes K_X(D)$ be a logarithmic connection
on $V$ singular over $D$. From \eqref{f4} it follows that the composition of homomorphisms
\begin{equation}\label{p2}
V\, \xrightarrow{\,\,\ \nabla^V\,\ }\, V\otimes K_X(D) \, \longrightarrow\,
(V\otimes K_X(D))_y \,\stackrel{\sim}{\longrightarrow}\, V_y
\end{equation}
is ${\mathcal O}_X$--linear; the isomorphism $(V\otimes K_X(D))_y\,\stackrel{\sim}{\longrightarrow}
\, V_y$ in \eqref{p2} is given by the isomorphism in \eqref{pa}, and the homomorphism
$V\otimes K_X(D) \, \longrightarrow\, (V\otimes K_X(D))_y$ is the restriction map. Therefore,
the composition of homomorphisms in \eqref{p2} produces a $\mathbb C$--linear homomorphism
$$
{\rm Res}({\nabla}^V,\,y)\, :\, V_y\, \longrightarrow\, V_y\, ,
$$
which is called the \textit{residue} of ${\nabla}^V$ at $y$; see \cite{De}.

Let $E_*\,:=\, \left(E,\,
\{\{E^j_i\}_{j=1}^{n_i}\}_{i=1}^\ell, \,\{\{\alpha_{i,j}\}_{j=1}^{n_i}\}_{i=1}^\ell\right)$ be a parabolic vector bundle. A \textit{quasiparabolic connection} on $E_*$ is a logarithmic connection ${\nabla}$ on $E$, singular over
$D$, such that $\text{Res}(\nabla,\,x_i)(E^j_i)\, \subset\, E^j_i$ for all $1\,\leq\, j\,\leq\, n_i$,
$1\,\leq\, i\, \leq\, \ell$ (see \eqref{e1}).

We will also consider the special class of quasiparabolic connections $\nabla$ on $E_*$ satisfying the extra
condition that the endomorphism of $E^j_i/E^{j+1}_i$ induced by
$\text{Res}(\nabla,\,x_i)$ coincides with multiplication by the parabolic weight $\alpha_{i,j}$ for all $1\,\leq\,
j\,\leq\, n_i$, $1\,\leq\, i\, \leq\, \ell$ (see \cite[Section~2.2]{BL}). This special class of quasiparabolic connections
$\nabla$ will be called \textit{parabolic connections}.

Given two quasiparabolic connections $(E_*,\,\nabla)$ and $(E_*',\,\nabla')$, their tensor product has
a natural quasiparabolic connection given by $(E_*,\,\nabla)\otimes (E'_*,\,\nabla')
\,=\, (E_*\otimes E_*',\, \nabla \otimes\text{Id}_{E'_*} + \text{Id}_{E_*}\otimes \nabla')$.
The dual $E_*^*$ also has a naturally induced quasiparabolic connection $\nabla^*\,:\,
E^*\,\longrightarrow\, E^*\otimes K_X(D)$ defined as
$$\langle \nabla u,\, v^* \rangle + \langle u, \,\nabla^* v^*\rangle \,=\, d\langle u,v^* \rangle$$
for each local section $u$ of $E$ and local section $v^*$ of $E^*$.

Now, let $L_*$ be a parabolic line bundle of parabolic degree $0$. Let $\{\beta_i\}$ be its
parabolic weights. As mentioned before, there exists a unique unitary logarithmic connection $D_{L}\,:\, L
\,\longrightarrow\, L\otimes K_X(D)$ such that ${\rm Res}(D_L,\,x_i)\,=\,\beta_i$
for all $1\, \leq\, i\, \leq\, \ell$ (see \cite{De}, \cite{Ka}). 

Take a parabolic vector bundle $E_*$ equipped with a 
parabolic symplectic or a parabolic orthogonal structure
$\phi\, :\, E_*\otimes E_* \longrightarrow L_*$. A quasiparabolic connection $\nabla \, : \, E
\,\longrightarrow\, E\otimes K_X(D)$ on $E_*$ is said to be \textit{compatible} with the
parabolic symplectic or parabolic orthogonal structure if the
map $\phi$ is flat (equivalently integrable) with respect to the quasiparabolic connection
$\nabla\otimes \text{Id}_{E_*} + \text{Id}_{E_*} \otimes \nabla$ induced on $E_*\otimes E_*$ by $\nabla$
and the unique unitary parabolic connection $D_L$ on $L_*$.
This condition means that for each pair of local sections $u,\, v$ of $E$,
\begin{equation}\label{eq:compconn}
(\phi\otimes \text{Id}_{K_X(D)}) ( \nabla(u) \otimes v) + (\phi\otimes \text{Id}_{K_X(D)})(u\otimes \nabla(v))
\,=\, D_L ( \phi(u\otimes v)).
\end{equation}

Observe that $\nabla$ is compatible with $\phi$ if and only if the isomorphism
$\widehat\phi\, : \, E_*\longrightarrow E_*^* \otimes L_*$ (see \eqref{e5})
is an isomorphism of quasiparabolic connections between $\nabla$ and the natural quasiparabolic
connection $\nabla^* \otimes \text{Id}_L+\text{Id}_{E_*^*} \otimes D_L$ on
$E_*^*\otimes L_*$ induced by $\nabla$ and $D_L$. In other words, $\nabla$ and $\phi$ are compatible
if the following diagram is commutative
\begin{eqnarray}
\label{eq:compconn2}
\xymatrixcolsep{6pc}
\xymatrix{
E \ar[r]^{\widehat{\phi}} \ar[d]_{\nabla} & E^*\otimes L \ar[d]^{\nabla^* \otimes \text{Id}_L +
\text{Id}_{E^*} \otimes D_L}\\
E\otimes K_X(D) \ar[r]^{\widehat\phi \otimes \text{Id}_{K_X(D)}} & E^*\otimes L \otimes K_X(D)\, .
}
\end{eqnarray}

A \textit{quasiparabolic connection} on $(E_*,\, \phi)$ is a quasiparabolic connection ${\nabla}$ on $E_*$, 
singular over $D$, which is compatible with $\phi$. As before, we say that $\nabla$ is a \textit{parabolic 
connection} on $(E_*,\, \phi)$ if $\nabla$ is a parabolic connection on $E_*$ and it is a quasiparabolic
connection on $(E_*,\, \phi)$. So a quasiparabolic connection $\nabla$ on $(E_*,\, \phi)$ us a parabolic
connection if the residue of $\nabla$ at $x_i$ acts on $E^j_i/E^{j+1}_i$, $j\, \leq\, i\,
\leq\, n_i$, by multiplication with the parabolic weight $\alpha_{i,j}$.

\begin{proposition}\label{prop-1-con}
Let $(E_*,\, \phi,\, \nabla)$ be a parabolic symplectic (respectively, orthogonal)
bundle on $X$ with a compatible quasiparabolic connection $\nabla$.
\begin{enumerate}
\item The quasiparabolic symplectic (respectively, orthogonal) connection $(E_*,\, \phi,\, \nabla)$ is semistable
if and only if the quasiparabolic connection $(E_*,\, \nabla)$ is semistable.

\item If $\nabla$ is a parabolic connection on $(E_*,\, \phi)$, then $(E_*,\, \phi,\, \nabla)$ is semistable
if and only if $(E_*,\, \nabla)$ is semistable.

\item If $\nabla$ is a quasiparabolic connection on $(E_*,\, \phi)$, then $(E_*,\, \phi,\, \nabla)$ is polystable
if $(E_*,\, \nabla)$ is polystable.
\end{enumerate}
\end{proposition}

\begin{proof}
A quasiparabolic connection $(E_*,\, \nabla)$ admits a canonical Harder--Narasimhan filtration.
So the proof of the first statement of the proposition is identical to the proof of the first statement of
Proposition \ref{prop0}.

The second statement of the proposition follows from the first statement because a parabolic connection
on $(E_*,\, \phi)$ is a quasiparabolic on $(E_*,\, \phi)$.

The proof of the third statement is identical to the proof of the statement that
a parabolic symplectic (respectively, parabolic orthogonal) vector bundle $(E_*,\, \phi)$ is polystable
if and only if the parabolic vector bundle $E_*$ is polystable (see Proposition \ref{prop0}).
\end{proof}

\subsection{Higgs bundles}

Let $E_*$ be a parabolic vector bundle on $X$. A \textit{parabolic Higgs field} on $E_*$ is a holomorphic section
$$
\theta\ \in \ H^0(X,\, \text{ad}_P(E)\otimes K_X(D)).
$$
(see \eqref{e3} for $\text{ad}_P(E)$). We say that a parabolic Higgs field $\theta$ is \textit{strongly parabolic} if, moreover,
$$
\theta\ \in \ H^0(X,\, \text{ad}^0_P(E)\otimes K_X(D))
$$
(see \eqref{e3b} for $\text{ad}^0_P(E)$). If $(E_*,\,\theta)$ and $(E'_*,\,\theta')$ are two parabolic
Higgs bundles, then $\theta\otimes \text{Id}_{E'_*}+\text{Id}_{E_*}\otimes \theta'$ is a parabolic
Higgs field on $E_*\otimes E'_*$. Also, $\theta$ induces a parabolic Higgs field $\theta^*
\,=\, -\theta^t \, : \, E^*\longrightarrow E^*\otimes K_X(D)$ on the dual parabolic $E_*^*$
as follows:
$$\langle \theta(u),\, v^*\rangle + \langle u,\, \theta^*(v^*)\rangle \,=\, 0$$
(the two terms take values in $K_X(D)$).

Let $(E_*,\, \phi)$ be a parabolic symplectic or orthogonal vector bundle on $X$. We say that a parabolic
Higgs field $\theta\,:\, E_*\,\longrightarrow\, E_*\otimes K_X(D)$ is compatible with $\phi$ if
\begin{equation}
\label{eq:compatibleHiggs1}
(\phi\otimes \text{Id}_{K_X(D)}) ( \theta(u) \otimes v) + (\phi\otimes \text{Id}_{K_X(D)})(u\otimes \theta(v))
\,=\, 0.
\end{equation}
Observe that this is analogous to the compatibility condition in \eqref{eq:compconn} for connections.
Equivalently, $\theta$ is compatible with $\phi$ if and only if the isomorphism
$\widehat\phi \, : \, E_* \,\longrightarrow E_*^*\, \otimes L_*$ in \eqref{e5} induces an
isomorphism of parabolic Higgs fields between $\theta$ and $\theta^*\otimes \text{id}_L$, i.e.,
if the following diagram is commutative:
\begin{eqnarray}
\label{eq:compatibleHiggs2}
\xymatrixcolsep{6pc}
\xymatrix{
E \ar[r]^{\widehat{\phi}} \ar[d]_{\theta} & E^*\otimes L \ar[d]^{\theta^* \otimes \text{Id}_L}\\
E\otimes K_X(D) \ar[r]^{\widehat\phi \otimes \text{Id}_{K_X(D)}} & E^*\otimes L \otimes K_X(D)\, .
}
\end{eqnarray}
A \textit{parabolic symplectic (respectively, parabolic orthogonal) Higgs bundle} is a parabolic
symplectic (respectively, parabolic orthogonal) vector bundle equipped with a compatible parabolic
Higgs field. A \textit{strongly parabolic symplectic (respectively, orthogonal) Higgs bundle} is
a parabolic symplectic (respectively, orthogonal) vector bundle equipped with a compatible
strongly parabolic Higgs field.

Let us provide an alternative useful characterization of parabolic symplectic and orthogonal Higgs bundles. Using the isomorphism
$\widehat{\phi}$ in \eqref{e5}, we have the isomorphism
\begin{equation}\label{e7}
\text{End}(E_*)\ \cong \ E_*\otimes E^*_* \otimes L_* \otimes L_*^* \ \xrightarrow{\,\,\,{\rm Id}_{E_*}
\otimes \widehat{\phi}^{-1} \otimes \text{id}_{L^*},\,\,} \ E_*\otimes E_* \otimes L_*^* 
\end{equation}
$$
=\ (\text{Sym}^2(E_*) \otimes L_*^*) \oplus (\bigwedge\nolimits^2 E_* \otimes L_*^*),
$$
where $\text{Sym}^2(E_*)$ is the parabolic symmetric product and $\bigwedge\nolimits^2 E_*$ is the parabolic
exterior product.

\begin{proposition}
\label{prop:parabolicHiggsdef}
Let $(E_*,\,\theta)$ be a parabolic Higgs bundle. A parabolic symplectic structure $\phi\,:\,E_*\otimes E_*
\,\longrightarrow\, L_*$ is compatible with $\theta$ if and only if, under the isomorphism \eqref{e7},
$$\theta\, \in \, H^0\left (X,\, {\rm Sym}^2(E_*)\otimes L_*^*\otimes K_X(D) \right).$$
Similarly, a parabolic orthogonal structure $\phi$ is compatible with $\theta$ if and only if
$$\theta\, \in \, H^0\left (X,\, \left (\bigwedge\nolimits^2 E_*\right)\otimes L_*^*\otimes K_X(D) \right).$$
\end{proposition}

\begin{proof}
Let $\phi$ be a parabolic symplectic structure. From equation \eqref{eq:compatibleHiggs1} and the fact that $\phi$ is antisymmetric, a parabolic symplectic structure is compatible with $\theta$ if and only if
$$(\phi\otimes \text{Id}_{K_X(D)})(u\otimes \theta(v))\,=\,- (\phi\otimes \text{Id}_{K_X(D)})
(\theta(u)\otimes v) \,=\, (\phi\otimes \text{Id}_{K_X(D)})(v\otimes \theta(u))$$
for each pair of local sections $u$ and $v$ of $E$. Thus, the map $(\phi \otimes \text{Id}_{K_X(D)})
\circ (\text{Id}_E \otimes \theta) \,:\, E_*\otimes E_* \,\longrightarrow\, L_* \otimes K_X(D)$
is symmetric; observe that this map can be rewritten as a contraction
$$(\phi\otimes \text{Id}_{K_X(D)})(u\otimes \theta(v)) \,=\,
 \langle \widehat\phi(u),\, \theta(v)\rangle .$$
Composing with the isomorphism map $(\widehat\phi^{-1} )^{\otimes 2}\,:\, E_*^* \otimes E_*^* \otimes L_*^2
\,\longrightarrow\, E_*\otimes E_*$, we obtain a symmetric map
\begin{equation}\label{j1}
E_*^*\otimes E_*^* \otimes L_*^2 \ \longrightarrow\ L_*\otimes K_X(D)
\end{equation}
which defines a section
$$\widetilde{\theta} \, \in \, H^0\left(X,\, \text{Hom}(E_*^*\otimes E_*^* \otimes L_*^2, L_*\otimes
K_X(D))\right)\, =\, H^0\left (X,\, E_*\otimes E_* \otimes L_*^* \otimes K_X(D) \right).$$
As the map in \eqref{j1} is symmetric, we have $\widetilde{\theta} \, \in \, H^0\left (X,\, \text{Sym}^2(E_*)
\otimes L_*^* \otimes K_X(D) \right)$. Locally, the expression of $\widetilde{\theta}$ is given by
$$\widetilde{\theta}(u^*,\,v^*)\ =\ \langle u^*,\, \theta( \phi^{-1}(v^*))\rangle$$
so, by construction, the section $\widetilde{\theta}$ corresponds to $\theta \,\in\,
H^0(X,\, \text{End}(E_*)\otimes K_X(D))$ under the isomorphism in \eqref{e7}.

The proof for parabolic orthogonal Higgs bundles is completely analogous; the map $(\phi \otimes 
\text{Id}_{K_X(D)}) \circ (\text{id}_E \otimes \theta)$ is antisymmetric in this case because $\phi$ 
is now symmetric.
\end{proof}

Let $(E_*,\, \phi, \, \theta)$ be a parabolic symplectic or orthogonal Higgs bundle. It is called
\textit{stable} (respectively, \textit{semistable}) if
$$
\text{par-deg}(F_*)\, \, <\,\, 0\ \ \, \text{(respectively, }\ \text{par-deg}(F_*)\, \, \leq\,\, 0\text{)}
$$
for every subbundle $0\, \not=\, F\, \subsetneq\, E$ such that $\theta (F)
\, \subset\, F\otimes K_X(D)$ and $\phi\big\vert_{X\setminus D}
((F\big\vert_{X\setminus D})\otimes (F\big\vert_{X\setminus D}))\,=\, 0$; as before, $F_*$ is the parabolic vector
bundle given by $F$ equipped with the parabolic structure induced by $E_*$.
A strongly parabolic symplectic or orthogonal Higgs bundle is called
\textit{stable} (respectively, \textit{semistable}) if it is \textit{stable} (respectively, \textit{semistable})
as a parabolic symplectic or orthogonal Higgs bundle.

We will now define polystable parabolic symplectic Higgs and polystable parabolic orthogonal Higgs bundles.

Take a parabolic Higgs vector bundle $(W_*,\, \theta_W)$ on $X$, where
$$
\theta_W\ \in\ H^0(X,\, \text{End}_P(E)\otimes K_X(D))
$$
(see \eqref{e2b} for $\text{End}_P(E)$). As seen in Section \ref{se2}, using \eqref{e6} the parabolic
vector bundle $W_*\oplus W^*_*$ has both symplectic and orthogonal structures. The Higgs field $\theta_W$ on
$W_*$ induces a Higgs field $\theta^*_W$ on the parabolic dual $W^*_*$. Now $W_*\oplus W^*_*$ has the Higgs field $\theta_W
\oplus -\theta^*_W$. For the above mentioned parabolic orthogonal (respectively, parabolic symplectic) structure on the parabolic vector bundle
$W_*\oplus W^*_*$, this $\theta_W\oplus -\theta^*_W$ is a Higgs field on the parabolic orthogonal (respectively, parabolic symplectic)
vector bundle.

A parabolic symplectic (respectively, parabolic orthogonal) Higgs bundle $(E_*,\, \phi, \, \theta)$ is
called \textit{polystable} if the following two conditions hold:
\begin{enumerate}
\item $(E_*,\, \phi,\, \theta)$ is semistable, and

\item $(E_*,\, \phi)\,=\, \bigoplus_{i=1}^n (V_{i,*},\, \phi_i,\, \theta_i)$, where each $(V_{i,*},\, \phi_i,\,
\theta_i)$ is either a stable parabolic symplectic (respectively, parabolic orthogonal) Higgs vector bundle or there is a polystable
parabolic Higgs vector bundle $(W_{i,*}, \theta_{W_i})$ of parabolic degree zero such that $(V_{i,*},\, \phi_i)$ is isomorphic to
$W^*_{i,*}\oplus W_{i,*}$ equipped with the above natural parabolic symplectic (respectively, parabolic orthogonal) pairing
and the Higgs field $\theta_{W_i} \oplus -\theta^*_{W_i}$ on the parabolic symplectic (respectively, parabolic orthogonal) vector bundle.
\end{enumerate}

\begin{proposition}\label{prop-1}
Let $$(E_*,\, \phi,\, \theta)$$ be a parabolic symplectic (respectively, parabolic orthogonal) Higgs bundle on $X$.
\begin{enumerate}
\item The parabolic symplectic (respectively, parabolic orthogonal) Higgs bundle $$(E_*,\, \phi,\, \theta)$$ is semistable
if and only if the parabolic Higgs bundle $(E_*,\, \theta)$ is semistable.

\item The parabolic symplectic (respectively, parabolic orthogonal) Higgs bundle $$(E_*,\, \phi,\, \theta)$$ is polystable
if and only if the parabolic Higgs bundle $(E_*,\, \theta)$ is polystable.
\end{enumerate}
\end{proposition}

\begin{proof}
The proof is exactly similar to the proof of Proposition \ref{prop0}. The required modifications are
straightforward. The details are omitted.
\end{proof}

\section{Pullback of parabolic bundles with symplectic and orthogonal structures}\label{section:pullback}

\subsection{Pullback of parabolic Higgs bundles}

Take $(X,\, D)$ as before. Let $Y$ be a compact connected Riemann surface and
\begin{equation}\label{a5}
f\, :\, Y\, \longrightarrow\, X
\end{equation}
a nonconstant holomorphic map. For each $x_i\, \in\, D$, let
\begin{equation}\label{a6}
f^{-1}(x_i)_{\rm red}\,=\, \{y_{i,1},\, \cdots,\, y_{i,b_i}\}\, \subset\, Y
\end{equation}
be the set-theoretic inverse image. The divisor $\sum_{j=1}^{b_i} y_{i,j}$ on $Y$ will also be denoted
by $f^{-1}(x_i)_{\rm red}$. Define the finite subset
\begin{equation}\label{a7}
B\,\,\,:=\,\, \,\bigcup_{i=1}^\ell f^{-1}(x_i)_{\rm red}\,\,\,=\,\,\, f^{-1}(D)_{\rm red}\,\, \subset\, \,Y .
\end{equation}
The divisor $\sum_{z\in B} z$ on $Y$ will also be denoted by $B$.

Given a parabolic vector bundle $E_*$ on $X$ with parabolic structure over $D$, there is a naturally associated
parabolic vector bundle $f^*E_*$ on $Y$ with parabolic structure over the divisor $B$ constructed in \eqref{a7}
(see \cite[Section 3]{AB}). For another parabolic vector bundle $F_*$ on $X$ with parabolic structure over $D$ we have
\begin{equation}\label{e8}
f^*(E_*\oplus F_*) \,=\, (f^*E_*)\oplus (f^*F_*),\ \ f^*(E_*\otimes F_*) \,=\, (f^*E_*)\otimes (f^*F_*),\ \
f^*(E^*_*) \,=\, (f^*E_*)^*
\end{equation}
(see \cite[p.~19559, Lemma 3.3]{AB} and \cite[p.~19560, Remark 3.4]{AB}).

\begin{proposition}\label{prop1}
Let $\phi\, :\, E_*\otimes E_*\, \longrightarrow\, L_*$ be an $L_*$--valued parabolic symplectic (respectively,
parabolic orthogonal) structure on $E_*$ (see Definition \ref{def1}). Then $f^*\phi$ is an $f^*L_*$--valued
parabolic symplectic (respectively, parabolic orthogonal) structure on the pulled back parabolic vector
bundle $f^*E_*$.

In particular, the pullback of an $\mathcal{O}_X$--valued parabolic orthogonal (respectively, symplectic)
structure on $E_*$ is an $\mathcal{O}_Y$--valued parabolic orthogonal (respectively, symplectic) structure
on $f^*E_*$.
\end{proposition}

\begin{proof}
{}From \eqref{e8} we have $f^*(E_*\otimes E_*)\,=\, (f^*E_*)\otimes (f^*E_*)$. Consequently, $f^*\phi$ is a homomorphism
$$
f^*\phi\, \, :\, \, f^*(E_*\otimes E_*)\,\,=\,\, (f^*E_*)\otimes (f^*E_*) \,\, \longrightarrow\,\, f^*L_*.$$
Next note that $f^*E^*_*\, =\, (f^*E_*)^*$ (see \eqref{e8}). Let
\begin{equation}\label{c1}
\widehat{f^*\phi}\ :\ f^*E_*\ \longrightarrow\ (f^*E_*)^* \otimes f^*L_* \ =\ f^*E^*_* \otimes f^*L_*
\end{equation}
be homomorphism of parabolic vector bundles given by the above pairing $f^*\phi\,:\,
(f^*E_*)\otimes (f^*E_*) \, \longrightarrow\, f^*L_*$. Since homomorphisms of parabolic
vector bundles produce homomorphisms of parabolic pullbacks, the isomorphism
$\widehat{\phi}\,:\, E_*\, \longrightarrow\, E^*_*\otimes L_*$ in \eqref{e5} pulls back to an isomorphism
$$f^* \widehat{\phi}\ :\ f^*E_*\ \longrightarrow\ f^*E^*_*\ =\ (f^*E_*)^*\otimes f^*L_*$$
(see \eqref{e8} for the above isomorphism $f^*E^*_*\, =\, (f^*E_*)^*$). On the other hand, this
homomorphism $f^* \widehat{\phi}$ clearly coincides with the homomorphism $\widehat{f^*\phi}$
in \eqref{c1}. Consequently, the fact that $f^* \widehat{\phi}$ is an isomorphism implies that
$\widehat{f^*\phi}$ is an isomorphism as well. This implies that $f^*\phi$ is an
$f^* L_*$--valued parabolic symplectic (respectively,
parabolic orthogonal) structure on the pulled back parabolic vector bundle $f^*E_*$.
\end{proof}

\begin{lemma}\label{lem6}\mbox{}
\begin{enumerate}
\item A parabolic symplectic (respectively, parabolic orthogonal) vector bundle $(E_*,\, \phi)$ on $X$
is semistable if and only if the pulled back parabolic symplectic (respectively, parabolic orthogonal) vector bundle
$(f^*E_*,\, f^*\phi)$ on $Y$ is semistable.

\item A parabolic symplectic (respectively, parabolic orthogonal) vector bundle $(E_*,\, \phi)$ on $X$
is polystable if and only if the parabolic symplectic (respectively, parabolic orthogonal) vector bundle
$(f^*E_*,\, f^*\phi)$ on $Y$ is polystable.
\end{enumerate}
\end{lemma}

\begin{proof}
The parabolic vector bundle $E_*$ on $X$ is semistable if and only if the parabolic vector bundle
$f^*E_*$ on $Y$ is semistable \cite[p.~19560, Lemma 3.5(2)]{AB}. This and Proposition \ref{prop0}(1)
combine together to give the first statement of the lemma.

The parabolic vector bundle $E_*$ on $X$ is polystable if and only if the parabolic vector bundle
$f^*E_*$ on $Y$ is polystable \cite[p.~19572, Theorem 5.6]{AB}. This and Proposition \ref{prop0}(2)
combine together to give the second statement.
\end{proof}

Now let $\theta$ be a Higgs field on the parabolic symplectic (respectively, parabolic orthogonal) vector bundle
$(E_*,\, \phi)$. The pullback $f^*\theta$ is a Higgs field on the pulled back parabolic vector bundle
$f^*E_*$ (see \cite[p.~19567, Proposition 5.1]{AB}).

\begin{lemma}\label{lem2}
The pullback $f^*\theta$ is a Higgs field on the parabolic symplectic (respectively, parabolic orthogonal)
vector bundle
$(f^*E_*,\, f^*\phi)$.
\end{lemma}

\begin{proof}
Since $\theta$ is a Higgs field on the parabolic symplectic (respectively, parabolic orthogonal) vector bundle
$(E_*,\, \phi)$, by Proposition \ref{prop:parabolicHiggsdef} we have
$$
\theta\ \in\ H^0(X,\ \text{Sym}^2(E_*)\otimes L_*^* \otimes K_X(D))
$$
(respectively, $\theta\, \in\, H^0(X,\, (\bigwedge^2 E_*)\otimes L_*^* \otimes K_X(D))$).

Note that
$$
f^*(K_X\otimes {\mathcal O}_X(D)) \ \subset\ K_Y\otimes {\mathcal O}_Y(B) \ =: \ K_Y(B),
$$
where $B$ is the divisor on $Y$ defined in \eqref{a7}. Also,
$$
f^*(\text{Sym}^2(E_*))\,=\, \text{Sym}^2(f^*E_*) \ \ \text{ and }\ \ f^*(\bigwedge\nolimits^2 E_*)\,=\,
\bigwedge\nolimits^2(f^*E_*)
$$
(see \eqref{e8}). Therefore, we have
$$
f^*\theta\ \in\ H^0(Y,\ \text{Sym}^2(f^* E_*)\otimes f^*L_*^* \otimes K_Y(B))
$$
(respectively, $f^*\theta\, \in\, H^0(Y,\, (\bigwedge^2 f^*E_*)\otimes f^* L_*^* \otimes K_Y(B))$). In view
of Proposition \ref{prop:parabolicHiggsdef}, from this it follows immediately that $f^*\theta$ is a Higgs
field on the parabolic symplectic (respectively, parabolic orthogonal) vector bundle $(f^*E_*,\, f^*\phi)$.
\end{proof}

\begin{lemma}\label{lem7}\mbox{}
\begin{enumerate}
\item A parabolic symplectic (respectively, parabolic orthogonal) Higgs bundle $(E_*,\, \phi, \theta)$ on $X$
is semistable if and only if the pulled back parabolic symplectic (respectively, parabolic orthogonal) Higgs
bundle $(f^*E_*,\, f^*\phi, \,f^*\theta)$ on $Y$ is semistable.

\item If a parabolic symplectic (respectively, parabolic orthogonal) Higgs bundle $(E_*,\, \phi, \theta)$ on $X$
is polystable, then the parabolic symplectic (respectively, parabolic orthogonal) Higgs bundle
$(f^*E_*,\, f^*\phi, \,f^*\theta)$ on $Y$ is polystable.
\end{enumerate}
\end{lemma}

\begin{proof}
The parabolic Higgs bundle $(E_*,\, \theta)$ is semistable if and only if the parabolic
Higgs bundle $(f^*E_*,\, f^*\theta)$ is semistable \cite[p.~19570, Lemma 5.4]{AB}. This and Proposition \ref{prop-1}(1)
combine together to give the first statement.

If the parabolic Higgs bundle $(E_*,\, \theta)$ is polystable then the parabolic
Higgs bundle $(f^*E_*,\, f^*\theta)$ is polystable; this follows immediately from
\cite[p.~19582, Theorem 7.3]{AB}. This fact and Proposition \ref{prop-1}(2)
combine together to give the second statement.
\end{proof}

\subsection{Pullback of parabolic connections}

Let $E_*$ be a parabolic vector bundle on $X$.
Let $\nabla$ be a quasiparabolic connection on a parabolic vector bundle $E_*$ on $X$. 
Then $f^*\nabla$ is a quasiparabolic connection on the pulled back parabolic vector bundle $f^*E_*$
(see \cite[p.~19572]{AB}).

\begin{lemma}\label{lem3}
Let $\phi$ be an $L_*$--valued parabolic symplectic (respectively, parabolic orthogonal) structure
on $E_*$. Let $\nabla$ be a quasiparabolic connection on the parabolic symplectic (respectively, parabolic
orthogonal) vector bundle $(E_*,\, \phi)$. Then the pulled back connection $f^*\nabla$ is a quasiparabolic
connection on the pulled back $f^*L_*$--valued parabolic symplectic (respectively, parabolic orthogonal)
vector bundle $(f^*E_*,\, f^*\phi)$ on $Y$.
\end{lemma}

\begin{proof}
As before, $D_{L}\,:\, L\,\longrightarrow\, L\otimes K_X(D)$ is the
unique unitary logarithmic connection on $L$ such that ${\rm Res}(D_L,\,x_i)\,=\,\beta_i$
for all $1\, \leq\, i\, \leq\, \ell$. Notice that the pullback connection $f^*D_L \, : \, f^*L_*
\longrightarrow f^*L_* \otimes K_Y(B)$ is the unique unitary parabolic connection on $f^*L_*$.
Thus, $D_{f^*L}\,=\,f^*D_L$, where $D_{f^*L}$ is the unique unitary parabolic connection on $f^*L_*$.
The restriction $(f^*E_*)\big\vert_{Y\setminus B}$ coincides with the usual pullback
$f^*(E\big\vert_{X\setminus D})$, where $E$ is the vector bundle on $X$ underlying the
parabolic vector bundle $E_*$. The parabolic symplectic (respectively, parabolic orthogonal)
structure $\phi\big\vert_{X\setminus D}$ is an usual symplectic (respectively, orthogonal)
structure on the vector bundle $E\big\vert_{X\setminus D}$. The restriction
$(f^*\phi)\big\vert_{Y\setminus B}$ is the usual pullback of $\phi\big\vert_{X\setminus D}$.

The connection $\nabla\big\vert_{X\setminus D}$ preserves the usual symplectic (respectively,
orthogonal) form $\phi\big\vert_{X\setminus D}$ on $E\big\vert_{X\setminus D}$. The pullback
$(f^*\nabla)\big\vert_{Y\setminus B}$
is the usual pullback of $\nabla\big\vert_{X\setminus D}$. We note that
$(f^*\nabla)\big\vert_{Y\setminus B}$ preserves $(f^*\phi)\big\vert_{Y\setminus B}$ because
$\nabla\big\vert_{X\setminus D}$ preserves the symplectic (respectively,
orthogonal) form $\phi\big\vert_{X\setminus D}$. 
Since $f^*\nabla$ is a quasiparabolic connection on the pulled back parabolic vector bundle $E_*$
(see \cite[p.~19572]{AB}), and
$(f^*\nabla)\big\vert_{Y\setminus B}$ preserves $(f^*\phi)\big\vert_{Y\setminus B}$, it follows
that $f^*\nabla$ is a quasiparabolic connection on the parabolic
symplectic (respectively, parabolic orthogonal) vector bundle $(f^*E_*,\, f^*\phi)$.
\end{proof}

\begin{lemma}\label{lem7-con}\mbox{}
\begin{enumerate}
\item A parabolic symplectic (respectively, parabolic orthogonal) connection $(E_*,\, \phi, \nabla)$ on $X$
is semistable if and only if the pulled back parabolic symplectic (respectively, parabolic orthogonal) connection
$(f^*E_*,\, f^*\phi, \,f^*\nabla)$ on $Y$ is semistable.

\item If a parabolic symplectic (respectively, parabolic orthogonal) connection $(E_*,\, \phi, \nabla)$ on $X$
is polystable, then the parabolic symplectic (respectively, parabolic orthogonal) connection
$(f^*E_*,\, f^*\phi, \,f^*\nabla)$ on $Y$ is polystable.
\end{enumerate}
\end{lemma}

\begin{proof}
The proof of the first statement is completely analogous to the proof of the first statement of
Lemma \ref{lem7}, using the second statement of Proposition \ref{prop-1-con}.

The proof of the second statement is completely analogous to the proof of the second statement of
Lemma \ref{lem7}, using the third statement of Proposition \ref{prop-1-con}.
\end{proof}

\section{Direct image of parabolic bundles with symplectic and orthogonal structures}
\label{section:pushforward}

\subsection{Direct image of parabolic Higgs bundles}

Let $Z$ be a compact connected Riemann surface and
\begin{equation}\label{e9}
\Phi\ :\ X\ \longrightarrow\ Z
\end{equation}
a nonconstant holomorphic map. Let $R\, \subset\, X$
be the ramification locus of $\Phi$. To clarify, we do not assume that $R$ and $D$ are disjoint.
For any point $x\,\in\, X$, let $m_x\, \geq\, 1$ be the multiplicity of $\Phi$ at $x$,
so $m_x\, \geq\, 2$ if and only if $x\, \in\, R$. Define the finite subset
\begin{equation}\label{e10}
\Delta\ =\ \Phi (R\cup D) \ \subset\ Z.
\end{equation}
The divisor $\sum_{\delta\in \Delta} \delta$ on $Z$ will also be denoted by $\Delta$.

For any parabolic vector bundle $E_*$ on $X$ with parabolic structure on $D$, the direct image
$\Phi_*E\, \longrightarrow\, Z$, where $E$ is the underlying vector bundle for $E_*$, has a
natural parabolic structure over the divisor $\Delta$ in \eqref{e10} \cite[Section~4]{AB}
(see also \cite[p.~19565, Lemma 4.1]{AB}). The vector bundle $\Phi_*E$ equipped with this
parabolic structure is denoted by $\Phi_*E_*$. From the construction of the parabolic vector
bundle $\Phi_*E_*$ it follows that
\begin{equation}\label{e11}
\Phi_*(E_*\oplus F_*) \ = \ (\Phi_* E_*)\oplus (\Phi_*F_*)\ \ \, \text{ and }\ \ \, 
\Phi_*(E^*_*)\ =\ (\Phi_*E_*)^*.
\end{equation}

Let us assume that $L_*\,\cong \,\Phi^* L'_*$ is the pullback of some parabolic line bundle $L'_*$ on $Z$
of parabolic degree zero. This is the case, for instance, if $L_*\,=\,\mathcal{O}_X\,=\,\Phi^*\mathcal{O}_Z$.
Let $\phi$ be an $L_*$--valued parabolic symplectic (respectively, parabolic orthogonal) structure on the parabolic
vector bundle $E_*$. Since $\Phi_*(E^*_*)\ =\ (\Phi_*E_*)^*$ (see \eqref{e11}), using the projection formula, the
isomorphism
$\widehat{\phi}\,:\, E_*\, \longrightarrow\, E^*_* \otimes L_*$ in \eqref{e5} induces an isomorphism
\begin{equation}\label{e12}
\Phi_*\widehat{\phi}\ :\ \Phi_* E_*\ \longrightarrow\ \Phi_*(E^*_*) \otimes L'_*\ =\ (\Phi_*E_*)^* \otimes L'_*.
\end{equation}

\begin{lemma}\label{lem4}
For the $L_*$--valued parabolic symplectic (respectively, parabolic orthogonal) structure on $\phi$ on $E_*$, the
homomorphism $\Phi_*\widehat{\phi}$ in \eqref{e12} gives an $L'_*$--valued parabolic symplectic
(respectively, parabolic orthogonal) structure on the parabolic vector bundle $\Phi_*E_*$.
\end{lemma}

\begin{proof}
Note that $\Phi_*\widehat{\phi}$ is anti-symmetric (respectively, symmetric) if $\widehat{\phi}$
is anti-symmetric (respectively, symmetric). The given condition that $\phi$ is a parabolic symplectic
(respectively, parabolic orthogonal) structure on $E_*$ implies that the homomorphism $\Phi_*\widehat{\phi}$
in \eqref{e12} is anti-symmetric (respectively, symmetric). Also, $\Phi_*\widehat{\phi}$ is
an isomorphism because $\widehat{\phi}$ is so. Therefore, it follows that $\Phi_*\widehat{\phi}$ gives a
parabolic symplectic (respectively, parabolic orthogonal) structure on $\Phi_*E_*$.
\end{proof}

This parabolic symplectic (respectively, parabolic orthogonal) structure on $\Phi_*E_*$
obtained in Lemma \ref{lem4} will be denoted by $\Phi_*\phi$.

\begin{lemma}\label{lem8}\mbox{}
\begin{enumerate}
\item A parabolic symplectic (respectively, parabolic orthogonal) vector bundle
$(E_*,\, \phi)$ is semistable if and only if the parabolic symplectic (respectively, parabolic orthogonal) vector bundle
$(\Phi_* E_*,\, \Phi_* \phi)$ is semistable.

\item A parabolic symplectic (respectively, parabolic orthogonal) vector bundle
$(E_*,\, \phi)$ is polystable if and only if the parabolic symplectic (respectively, parabolic orthogonal) vector bundle
$(\Phi_* E_*,\, \Phi_* \phi)$ is polystable.
\end{enumerate}
\end{lemma}

\begin{proof}
The parabolic vector bundle $E_*$ on $X$ is semistable if and only if the parabolic vector bundle $\Phi_* E_*$ on $Z$
is semistable \cite[p.~19567, Proposition 4.3]{AB}. This and Proposition \ref{prop0}(1)
combine together to give the first statement of the lemma.

The parabolic vector bundle $E_*$ on $X$ is polystable if and only if the parabolic vector bundle $\Phi_* E_*$ on $Z$
is polystable (see \cite[p.~19574, Proposition 5.7]{AB} and \cite[p.~19587, Theorem 7.5]{AB}). This and Proposition \ref{prop0}(2)
combine together to give the second statement.
\end{proof}

Now let $\theta$ be a Higgs field on the parabolic vector bundle $E_*$. Then
$\theta$ induces a Higgs field on the parabolic vector bundle $\Phi_* E_*$ (see \cite[p.~19576, (6.6)]{AB});
this induced Higgs field on $E_*$ is denoted by $\Phi_* \theta$.

\begin{lemma}\label{lem5}
Assume that $L_*\,=\,\Phi^*L'_*$. Let $\phi$ be an $L_*$--valued parabolic symplectic (respectively, parabolic orthogonal) structure on the parabolic vector
bundle $E_*$. If $\theta$ is a Higgs field on the parabolic symplectic (respectively, parabolic orthogonal) vector bundle
$(E_*,\, \phi)$, then it $\Phi_* \theta$ is an $L'_*$-valued Higgs field on the parabolic
symplectic (respectively, parabolic orthogonal) vector bundle $(\Phi_* E_*,\, \Phi_*\phi)$.
\end{lemma}

\begin{proof}
A Higgs field $\theta$ on the parabolic symplectic (respectively, parabolic orthogonal) vector
bundle $(E_*,\, \phi)$ is a section of $\text{Sym}^2(E_*) \otimes L^*_* \otimes K_X(D)$ (respectively,
$(\bigwedge ^2 E_*) \otimes L^*_* \otimes K_X(D)$). Note that $K_X\otimes{\mathcal O}_X(D)\, \subset\,
\Phi^*(K_Z\otimes {\mathcal O}_Z(\Delta))$. Therefore, $\Phi_* \theta$ gives an element of
$H^0(Z,\, \text{Sym}^2(\Phi_* E_*)\otimes (L'_*)^* \otimes K_Z\otimes{\mathcal O}_Z(\Delta))$
(respectively, $H^0(Z,\, \bigwedge^2(\Phi_* E_*)\otimes (L'_*)^* \otimes K_Z\otimes{\mathcal O}_Z(\Delta))$). From
this it follows that $\Phi_* \theta$ is a Higgs field on the $L'_*$--valued parabolic
symplectic (respectively, parabolic orthogonal) vector bundle $(\Phi_* E_*,\, \Phi_*\phi)$.
\end{proof}

\begin{lemma}\label{lem9}\mbox{}
\begin{enumerate}
\item A parabolic symplectic (respectively, parabolic orthogonal) Higgs bundle $(E_*,\, \phi,\, \theta)$ is
semistable if and only if the parabolic symplectic (respectively, parabolic orthogonal) Higgs bundle
$(\Phi_* E_*,\, \Phi_* \phi,\,\Phi_*\theta)$ is semistable.

\item If a parabolic symplectic (respectively, parabolic orthogonal) Higgs bundle $(E_*,\, \phi,\, \theta)$ is
polystable then the parabolic symplectic (respectively, parabolic orthogonal) Higgs bundle
$(\Phi_* E_*,\, \Phi_* \phi,\,\Phi_*\theta)$ is polystable.
\end{enumerate}
\end{lemma}

\begin{proof}
The parabolic Higgs bundle $(E_*,\, \theta)$ is semistable if and only if the parabolic
Higgs bundle $(\Phi_* E_*,\, \Phi_* \theta)$ is semistable \cite[p.~19577, Proposition 6.2]{AB}.
This and Proposition \ref{prop-1}(1) combine together to give the first statement of the lemma.

If the parabolic Higgs bundle $(E_*,\, \theta)$ is polystable then the parabolic
Higgs bundle $(\Phi_* E_*,\, \Phi_* \theta)$ is polystable \cite[p.~19587, Theorem 7.5]{AB}.
This and Proposition \ref{prop-1}(2) combine together to give the second statement.
\end{proof}

\subsection{Direct image of parabolic connections}

Let $\Phi\, :\, X\,\longrightarrow\, Z$ be the map in \eqref{e9}. As before, assume that $L_* \,=\,\Phi^*L'_*$ 
is the pullback of some parabolic line bundle $L'_*$ on $Z$ of parabolic degree zero. Let $D_{L'}$ be the unique 
unitary parabolic connection on $L'_*$. Then it is clear that $\Phi^*D_{L'}$ is the unique unitary parabolic 
connection $D_L$ on $L_*\,=\,\Phi^*D_{L'}$. Let $\nabla$ be a quasiparabolic connection on the parabolic vector 
bundle $E_*$ on $X$. Then $\nabla$ induces a quasiparabolic connection on the parabolic vector bundle $\Phi_* 
E_*$ (see \cite[Section 6.2]{AB}); this induced quasiparabolic connection on $\Phi_* E_*$ will be denoted by 
$\Phi_*\nabla$. Now let $\phi$ be a parabolic symplectic (respectively, parabolic orthogonal) structure on 
$E_*$. Assume that $\nabla$ is compatible with $\phi$. Then $\Phi_*\nabla$ is a quasiparabolic connection on the 
parabolic symplectic (respectively, parabolic orthogonal) vector bundle $(E_*,\, \phi)$.

\begin{lemma}\label{lem10}\mbox{}
\begin{enumerate}
\item A quasiparabolic symplectic (respectively, quasiparabolic orthogonal) connection
$$(E_*,\, \phi,\, \nabla)$$ is
semistable if and only if the quasiparabolic symplectic (respectively, quasiparabolic orthogonal) connection
$$(\Phi_* E_*,\, \Phi_* \phi,\,\Phi_*\nabla)$$ is semistable.

\item If a quasiparabolic symplectic (respectively, quasiparabolic orthogonal) connection
$$(E_*,\, \phi,\, \theta)$$ is
polystable then the quasiparabolic symplectic (respectively, quasiparabolic orthogonal) connection
$$(\Phi_* E_*,\, \Phi_* \phi,\,\Phi_*\theta)$$ is polystable.
\end{enumerate}
\end{lemma}

\begin{proof}
The proof is analogous to Lemma \ref{lem9} and it is a consequence of Proposition \ref{prop-1-con}
\cite[Proposition 6.4]{AB} and \cite[Theorem 7.5]{AB}.
\end{proof}

\section{Nonabelian Hodge correspondence}
\label{section:NAHC}

Finally, let us explore the compatibility between the previous construction and the Nonabelian Hodge Theory for 
noncompact curves \cite{Si}. In this section, given a curve $X$, let us denote by $NAHC_X$ the functor 
giving an equivalence of categories between the category of polystable quasiparabolic Higgs bundles of parabolic 
degree $0$ and the category of polystable quasiparabolic connections on $X$ of parabolic degree $0$ given by 
\cite[Main Theorem, p. 755]{Si}.

Let $L_*$ be a parabolic line bundle of parabolic degree $0$. Let $\{\beta_i\}$ be its parabolic weights. By 
\cite{Ka} and \cite{Si}, there exists a unique unitary logarithmic connection $D_{L}: L\longrightarrow L\otimes 
K_X(D)$ such that ${\rm Res}(D_L,x_i)=\beta_i$.

Then, the parabolic connection $(L_*,\,D_L)$ corresponds through $NAHC_X$ to the parabolic Higgs bundle $(L_*,\,
0)$.

\begin{theorem}
\label{thm:NAHC}
Let $(L_*,\,D_L)$ be a parabolic line bundle of parabolic degree $0$ with its associated unitary logarithmic
connection. Let $(E_*, \,\theta)$ be a parabolic Higgs bundle, and let $(E'_*, \,\nabla')$ be the flat
parabolic connection associated to $(E_*, \,\theta)$ through nonabelian Hodge Correspondence. Then, there exists
a one on one correspondence between the following two:
\begin{itemize}
\item Symplectic (respectively, orthogonal) structures $\phi\,:\,E_*\otimes E_* \,\longrightarrow\, L_*$
on $E_*$ compatible with $\theta$.

\item Symplectic (respectively, orthogonal) structures $\phi'\,:\, E'_* \otimes E'_* \,\longrightarrow\,
L_*$ which are compatible with $\nabla'$.
\end{itemize}
As a consequence, the Nonabelian Hodge Correspondence gives an equivalence between the categories of
polystable parabolic symplectic (respectively, orthogonal) Higgs bundles and polystable flat parabolic
symplectic (respectively, orthogonal) connections.
\end{theorem}

\begin{proof}
Since the Nonabelian Hodge Correspondence $NAHC_X$ is an equivalence of categories compatible with tensor 
products and duals \cite[Theorem 2]{Si}, it follows that $(E_*^*,\, \theta^*)$ is associated to 
$((E'_*)^*,\,(\nabla')^*)$ through $NAHC_X$, and $$(E_*^*\otimes L_*,\, \theta^*\otimes \text{Id}_{L_*} +
\text{Id}_{E^*}\otimes 0)\,=\, (E_*^*\otimes L_*,\, \theta^*\otimes \text{Id}_{L_*})$$ is associated to
$((E'_*)^* \otimes L_*, \, (\nabla')^*\otimes \text{Id}_L+\text{id}_{E^*} \otimes D_L)$ through $NAHC_X$.
By \eqref{eq:compatibleHiggs1}, 
$\phi$ is compatible with $\theta$ if and only if, $\phi$ is a homomorphism of parabolic Higgs bundles $\phi\, :\, 
(E_*,\,\theta) \otimes (E_*,\,\theta) \,\longrightarrow\, (L_*,\,0)$. As $NAHC_X$ is a functor, there must exist
a map of quasiparabolic connections $\phi'\,=\, NAHC_X(\phi)$ with
$$\phi' \ : \ NAHC_X((E_*,\,\theta)\otimes (E_*,\,\theta))\ =\ (E'_X,\,\nabla')\otimes (E'_X,\,\nabla')
$$
$$
\,\longrightarrow \ NAHC_X(L_*,\,0)\ =\ (L_*,\,D_L).$$
Moreover, by \eqref{eq:compatibleHiggs2}, $\phi$ is compatible with $\theta$ if and only if it induces an
isomorphism of quasiparabolic Higgs bundles $\widehat\phi \, :\, (E_*,\,\theta) \,\longrightarrow\, (E_*^* ,
\,\theta^*)\otimes (L_*,\,0)$, and, since $NAHC_X$ is an equivalence of categories, this happens if and only
if $\phi'$ induces an isomorphism of quasiparabolic connections $\widehat{\phi}' \, : \, ((E'_*)^*,\,
(\nabla')^*)\otimes (L_*,\,D_L)$. By \eqref{eq:compconn} and \eqref{eq:compconn2}, these conditions are equivalent to $\phi'$ being compatible with $\nabla'$.

On the other hand, $\phi$ is symmetric if and only if $\phi \circ \tau \,=\, \phi$, where $\tau$ is the natural
transposition (see Lemma \ref{lemma:transposition}). It is clear that $NAHC_X(\tau)\,=\,\tau$ and $NAHC_X$ is
an equivalence of categories, so $\phi$ is symmetric if and only if
$$\phi' \circ \tau \,=\, NAHC_X(\phi)\circ NAHC_X(\tau) \,=\, NAHC_X(\phi\circ \tau)\,=\, NAHC_X(\phi)\,=\,\phi'.$$
Thus, $\phi$ is symmetric if and only if $\phi'$ is symmetric. Analogously, $\phi$ is antisymmetric if and
only if $\phi'$ is antisymmetric.

As a consequence, $\phi$ is a parabolic orthogonal (respectively, symplectic) structure on $E_*$ compatible with 
$\theta$ if and only if $\phi'\,=\,NAHC_X(\phi)$ is a parabolic orthogonal (respectively, symplectic) structure on 
$E'_*$ compatible with $\nabla'$.

Finally, from Proposition \ref{prop-1-con} and Proposition \ref{prop-1}, the parabolic polystability of a
parabolic symplectic (respectively, orthogonal) Higgs bundles or a quasiparabolic symplectic (respectively,
orthogonal) connection is equivalent to the polystability of its underlying quasiparabolic Higgs bundle or
quasiparabolic connection (forgetting the symplectic or orthogonal structure), so the Nonabelian Hodge
Correspondence between polystable parabolic Higgs bundles and polystable quasiparabolic connections of
parabolic degree 0 (see \cite{Si}) induces the desired equivalence of categories.
\end{proof}

\begin{proposition}
The correspondence from Theorem \ref{thm:NAHC} is compatible with the pullbacks and direct images described 
through Section \ref{section:pullback} in the following sense:

Let $L_*$ be a parabolic line bundle of parabolic degree $0$ on $X$. Let $f\,:\,Y\,\longrightarrow\, X$ be a 
nonconstant map of Riemann surfaces. Let $(E_*,\, \phi,\, \theta)$ be a polystable quasiparabolic $L_*$--valued 
orthogonal (respectively, symplectic) Higgs bundle, and let $(E'_*,\,\phi',\,\nabla')$ be its 
associated quasiparabolic orthogonal (respectively, symplectic) connection through Nonabelian 
Hodge Correspondence. Then $(f^*E_*,\, f^*\phi,\, f^*\theta)$ and $(f^*E'_*,\,f^*\phi',\,f^*\nabla')$
are polystable and they are associated by Nonabelian Hodge Correspondence.

Similarly, let $\Phi\, : \, X \,\longrightarrow\, Z$ be a nonconstant map. Assume that $L_*\,=\,\Phi^*L'_*$ for some 
parabolic line bundle $L_*'$ on $Z$ of parabolic degree zero. Then $(\Phi_*E_*,\, \Phi_*\phi,\, \Phi_* \theta)$ and 
$(\Phi_*E'_*,\,\Phi_*\phi',\,\Phi_*\nabla')$ are polystable and are associated through Nonabelian Hodge 
Correspondence.
\end{proposition}

\begin{proof}
The pullback and direct image of the given parabolic orthogonal (respectively, symplectic) Higgs bundle $(E_*,\, \phi,\, 
\theta)$ are polystable parabolic orthogonal (respectively, symplectic) Higgs bundles by Lemma \ref{lem2}, Lemma 
\ref{lem7}, Lemma \ref{lem4} and \ref{lem8}.

Similarly, the pullback and direct image of the given quasiparabolic orthogonal (respectively, symplectic) Higgs bundle 
are polystable quasiparabolic orthogonal (respectively, symplectic) connections by Lemma \ref{lem3}, Lemma 
\ref{lem7-con}, Lemma \ref{lem9} and \ref{lem10}.

By \cite[Theorem 7.3]{AB} and \cite[Theorem 7.5]{AB}, the functor $NAHC_X$ is compatible with parabolic 
pullbacks and direct images of quasiparabolic Higgs bundles and quasiparabolic connections, so 
$(f^*E_*,\,f^*\theta)$ is associated to $(f^*E'_*,\,f^*\nabla')$ through the functor $NAHC_Y$ and 
$(\Phi_*E_*,\,\Phi_*\theta)$ is associated to $(\Phi_*E'_*,\,\Phi_*\nabla')$ through the functor $NAHC_Z$.

The argument in Theorem \ref{thm:NAHC} then implies that the functor $NAHC_X$ sends the isomorphism of parabolic 
Higgs bundles $$f^*\widehat\phi \, : \, (f^*E_*,\,f^*\theta) \,\stackrel{\sim}{\longrightarrow}\, 
(f^*E_*^*,\,f^*\theta^*)\otimes (f^*L_*,\,0)$$ to the isomorphism of quasiparabolic connections $$f^*\widehat\phi' 
\, : \, (f^*E'_*,\,f^*\nabla') \,\stackrel{\sim}{\longrightarrow}\, (f^*(E'_*)^*,f^*(\nabla')^*) \otimes 
(f^*L_*,f^*D_L),$$ so it maps $f^*\phi $ to $f^*\phi'$. Analogously, it sends the isomorphism of parabolic Higgs 
bundles $$\Phi_*\widehat\phi \, : \, (\Phi_*E_*,\,\Phi_*\theta) \,\stackrel{\sim}{\longrightarrow}\, 
(\Phi_*E_*^*,\,\Phi_*\theta^*)\otimes (L'_*,\,0)$$ to the isomorphism of quasiparabolic connections 
$$\Phi_*\widehat\phi' \, : \, (\Phi_*E'_*,\,\Phi_*\nabla') \,\stackrel{\sim}{\longrightarrow}\,
(\Phi_*(E'_*)^*,\,\Phi_*(\nabla')^*) \otimes (L'_*,\,D_{L'}),$$ so it maps $\Phi_*\phi$ to $\Phi_*\phi'$.
\end{proof}

\section*{Acknowledgements}

D.A. was supported by grants PID2022-142024NB-I00 and RED2022-134463-T funded by 
MCIN/AEI/10.13039/501100011033. I.B. is partially supported by a J. C. Bose Fellowship (JBR/2023/000003).


\end{document}